\documentclass[a4paper,12pt]{article}
\usepackage{latexsym, amsmath, amsfonts, amscd, amssymb, verbatim}
\usepackage{graphicx}	
\usepackage{xcolor}

\def\tto{\;{\lower 1pt \hbox{$\rightarrow$}}\kern -10pt
	\hbox{\raise 2pt \hbox{$\rightarrow$}}\;}

\begin{document}
	\pagestyle{myheadings}
	\newtheorem{Theorem}{Theorem}[section]
	\newtheorem{Proposition}[Theorem]{Proposition}
	\newtheorem{Remark}[Theorem]{Remark}
	\newtheorem{Lemma}[Theorem]{Lemma}
	\newtheorem{Corollary}[Theorem]{Corollary}
	\newtheorem{Definition}[Theorem]{Definition}
	\newtheorem{Example}[Theorem]{Example}
	\renewcommand{\theequation}{\thesection.\arabic{equation}}
	\normalsize
	\setcounter{equation}{0}

\title{\bf Generalized Polyhedral DC Optimization Problems}
 
\author{Vu Thi Huong\footnote{Institute of Mathematics, Vietnam Academy of Science and Technology, 18 Hoang Quoc Viet, Hanoi 10072, Vietnam; email: vthuong@math.ac.vn.},\ \, Duong Thi Kim Huyen\footnote{ORLab, Faculty of Computer Science, Phenikaa University, Hanoi 12116, Vietnam; Email: huyen.duongthikim@phenikaa-uni.edu.vn.},\ \, and Nguyen Dong Yen\footnote{Institute of Mathematics, Vietnam Academy of Science and Technology, 18 Hoang Quoc Viet, Hanoi 10307, Vietnam; email: ndyen@math.ac.vn; Corresponding author.}
	}\maketitle
\date{}%\small\today}

\medskip
\begin{quote}
\noindent {\bf Abstract.}\ The problem of minimizing the difference of two lower semicontinuous, proper, convex functions (a DC function) on a nonempty closed convex set in a locally convex Hausdorff topological vector space is studied in this paper. The focus is made on the situations where either the second component of the objective function is a generalized polyhedral convex function or the first component of the objective function is a generalized polyhedral convex function and the constraint set is generalized polyhedral convex. Various results on optimality conditions, the local solution set, the global solution set, and solution algorithms via duality are obtained. Useful illustrative examples are considered.

\medskip
\noindent {\bf Mathematics Subject Classification (2010).}\ 91A05, 91A10, 90C05, 49J53.

\medskip
\noindent {\bf Key Words.}\ Difference of convex functions, generalized polyhedral convex function, generalized polyhedral convex set, optimality condition, local solution set, global solution set, DCA scheme.
 
\end{quote}
 
\section{Introduction}\label{sect_Introduction}
\markboth{\sc v. t. huong, d. t. k. huyen, and n.~d.~yen}{\sc  v. t. huong, d. t. k. huyen, and n.~d.~yen}
\setcounter{equation}{0}

A \textit{polyhedral convex set} (a convex polyhedron in brief) in $\mathbb R^n$ or, more generally, in a finite-dimensional normed space $X$, is the intersection of finitely  many closed half-spaces.
Since the intersection of an empty family of closed half-spaces is  
whole space by convention, empty set and the whole space are  special polyhedra. For every given convex polyhedron, there exist  
a finite number of points and a finite number of directions such that the polyhedron can be represented as the sum of the convex hull of those points and the convex cone generated by those directions. The converse is also true. This fundamental result is presented in~\cite[Theorem~19.1]{Rock_book_1970} and is attributed~\cite[p.~427]{Rock_book_1970} primarily to H.~Minkowski and H.~Weyl. The theorem allows one to easily prove fundamental solution existence theorems in linear programming. It is also the basis of Contesse's proofs of the necessary and
sufficient second-oder conditions for a local solution and for a locally unique solution in quadratic programming in~\cite{Contesse_1980}.

Representation formulas in the spirit of~\cite[Theorem~19.1]{Rock_book_1970} were 
obtained by Zheng~\cite{Zheng_2009} for \textit{convex polyhedra} and \textit{generalized convex polyhedra} in Banach spaces. Luan and Yen~\cite{Luan_Yen_2020} have shown that similar representation formulas are valid for convex polyhedra and generalized convex polyhedra in locally convex Hausdorff topological vector spaces.
Applications of the representation formulas to obtaining
solution existence theorems for generalized linear programming problems and generalized linear vector optimization
problems can be found in~\cite{Luan_Yen_2020}.

Generalized polyhedral convex sets, generalized polyhedral
convex functions on Hausdorff locally convex topological vector spaces, and the related constructions such as sum of sets, sum
of functions, directional derivative, infimal convolution, normal
cone, conjugate function, subdifferential have been studied by Luan et al.~\cite{Luan_Yao_Yen_2018}.

\textit{Generalized polyhedral convex optimization problems} in locally convex Hausdorff topological vector spaces have been investigated systematically by Luan and Yao~\cite{Luan_Yao_JOGO_2019}, where  solution existence theorems, necessary and sufficient optimality conditions, weak and strong duality theorems are proved. 

As far as we know, the problem of minimizing the difference of two lower semicontinuous, proper, convex functions (a DC function) on a nonempty closed convex set in a locally convex Hausdorff topological vector space has not been considered until now. Our aim is to study this problem, making focus on the situations where either the second component of the objective function is a generalized polyhedral convex function or the first component of the objective function is a generalized polyhedral convex function and the constraint set is generalized polyhedral convex. Note that, in finite-dimensional settings, the problem has been considered by Pham Dinh and Le Thi~\cite{Tao_An_AMV97}, Polyakova~\cite{Polyakova_2011}, Hang and Yen~\cite{HangNTV_Yen_2016}, vom Dahl and~L\"ohne~\cite{vom Dahl_Lohne_2020} from different points of view. 

The ingenious \textit{DC algorithms} (usually called~\textit{DCA} for brevity), which allow one to decompose the given nonconvex optimization problem into that of solving two sequences of convex programs in the primal space and the dual space respectively, were suggested and employed for polyhedral DC programming in~\cite{Tao_An_AMV97}. For comprehensive surveys of the successful development and numerous applications of DCA, the interested reader is referred to~\cite{An_Tao_2017,An_Tao_2024}. Necessary and sufficient conditions for a global solution of an unconstrained polyhedral DC program via hypodifferentials of polyhedral convex functions were given in~\cite[Theorems~5 and~6]{Polyakova_2011}.  Optimality conditions via the subdifferential in the sense of convex analysis, the Fr\'echet subdifferential, and the Mordukhovich subdifferential for unconstrained and linearly constrained polyhedral DC programs, and the relationships between these conditions, along with the existence and computation of descent and steepest descent directions, were investigated in~\cite{HangNTV_Yen_2016}. The solution existence and a method to globally solve polyhedral DC optimization problems via concave minimization were addressed in~\cite{vom Dahl_Lohne_2020}.

We will establish not only fundamental qualitative properties of generalized polyhedral DC optimization problems, but also remarkable features of DCA applied to these problems. Theorems~\ref{thm_structure1} and~\ref{thm_structure2} on the structure of the local solution set and the structure of the global solution set, Theorems~\ref{thm_const_value} and~\ref{thm_const_value_2} on the constancy of the objective function on each connected component of the solution set in question, as well as Theorems~\ref{Thm_DCA2} and \ref{Thm_DCA2a} on the cyclic behavior of DCA iterative sequences are new even in finite dimensions. Note that many proofs herein (for example, the proofs of Theorems~\ref{thm_structure1} and~\ref{thm_const_value}) rely on refined arguments and special constructions.

The organization of this paper is as follows. In Section~\ref{sect_Preliminaries}, we present basic notions and several auxiliary results on generalized polyhedral convex sets and functions, generalized polyhedral DC optimization, conjugate functions and duality. A series of new results on optimality conditions, the local solution set, the global solution set of generalized polyhedral DC optimization problems on Hausdorff locally convex topological vector spaces are obtained in Sections~\ref{sect_3} and~\ref{sect_4}. Section~\ref{sect_5} is devoted to DCA schemes, which are solution algorithms via duality. Useful illustrative examples are considered in Sections~\ref{sect_3}--~\ref{sect_5}. Section~\ref{sect_Conclusions} gives some concluding remarks.

\section{Preliminaries}\label{sect_Preliminaries}
\markboth{\sc v. t. huong, d. t. k. huyen, and n.~d.~yen}{\sc  v. t. huong, d. t. k. huyen, and n.~d.~yen}
\setcounter{equation}{0}

From now on, if not otherwise stated, $X$ is a {\it locally convex Hausdorff topological vector space}. We denote by $X^*$ the dual space of $X$ and by $\langle x^*, x \rangle$ the value of $x^* \in X^*$ at $x \in X$. It is assumed that $X^*$ is equipped with the weak$^*$ topology. For any subset $\Omega\subset X$, ${\rm int}\Omega$ denotes the interior of $\Omega$ and ${\rm co}\Omega$ stands for the convex hull of $\Omega$. By ${\rm cone}\,\Omega$ we denote cone generated by~$\Omega$, that is, ${\rm cone}\,\Omega=\{tx\mid t\geq 0,\; x\in\Omega\}$. The set of nonnegative integers is denoted by~$\mathbb N$. The convention $(+\infty)-(+\infty)=+\infty$ will be used in the sequel.

\subsection{Generalized polyhedral convex sets and functions}\label{subsect_pdco}

\begin{Definition}{\rm (See \cite[p.~133]{Bonnans_Shapiro_2000})\label{Def_gpcs}
		A subset $C\subset X$ is said to be a \textit{generalized polyhedral convex set}, or a \textit{generalized convex polyhedron}, if there exist $x^*_k \in X^*$, $\alpha_k \in \mathbb R$, $k=1,\dots,p$, and a closed affine subspace $L \subset X$, such that 
		\begin{equation}\label{eq_def_gpcs}
			C=\big\{ x \in X \mid x \in L,\ \langle x^*_k,x\rangle \leq \alpha_k,\ k=1,\ldots,p\big\}.
		\end{equation} 
		If the set $C$ can be represented in the form \eqref{eq_def_gpcs} with $L=X$, then we say that it is a \textit{polyhedral convex set}, or a \textit{convex polyhedron}.}
\end{Definition}

If $X$ is a finite-dimensional locally convex Hausdorff topological vector space (in particular, $X=\mathbb R^n$), then the notions of generalized convex polyhedron and convex polyhedron coincide.

Let $C$ be given by \eqref{eq_def_gpcs}. By~\cite[Remark 2.196]{Bonnans_Shapiro_2000}, there exists a continuous surjective linear mapping $A:X\to Y$ and a vector $y\in Y$ such that $L=\{x\in X\mid A(x)=y\}$. So, we have
	\begin{equation}\label{equivalence_gpcs}
		C=\{x\in X\mid A(x)=y,\ \langle x^*_k,x\rangle \leq \alpha_k,\ k=1,\dots,p\}.
	\end{equation}
For a function $f$ from a locally convex Hausdorff topological vector space $X$ to the extended real line $\overline{\mathbb{R}}:=\mathbb{R} \cup \{+\infty\}\cup \{-\infty\}$, one defines the \textit{effective domain} and the \textit{epigraph} of $f$, respectively, by setting ${\rm dom}f=\{x \in X \mid f(x) < +\infty\}$ and $${\rm epi}f=\big\{(x, \alpha) \in X \times \mathbb{R} \mid x \in {\rm dom}f,\ f(x) \leq \alpha \big\}.$$  If ${\rm dom}f$ is nonempty and $f(x) > - \infty$ for all $x \in X$, then $f$ is said to be \textit{proper}. We say that $f$ is  \textit{convex} if  ${\rm epi}f$ is a convex set in $X \times \mathbb{R}$. 

\begin{Definition} {\rm (See~\cite[Definition~3.1]{Luan_Yao_Yen_2018}) Let $X$ be a locally convex Hausdorff topological vector space. A function $f:X\to\overline{\mathbb{R}}$ is called \textit{generalized polyhedral convex} (resp., \textit{polyhedral convex}) if ${\rm epi}\,f$ is a generalized polyhedral convex set (resp., a polyhedral convex set) in $X\times\mathbb R$.}
\end{Definition}

It turns out that a generalized polyhedral convex function (resp., a polyhedral convex function) can be characterized as the maximum of a finite family of continuous affine functions on a certain generalized polyhedral convex set (resp., a polyhedral convex set).

\begin{Lemma}\label{lem_rep_gpcf} {\rm (See \cite[Theorem~3.2]{Luan_Yao_Yen_2018})} Suppose that $\psi:X\to\overline{\mathbb{R}}$ is a proper function. Then $\psi$ is generalized polyhedral convex (resp., polyhedral convex) if and only if ${\rm dom}\psi$ is a generalized polyhedral convex set (resp., a polyhedral convex set) in $X$ and there exist $w_k^* \in X^*$, $\gamma_k \in \mathbb{R}$, for $k=1,\ldots,r$, such that
	\begin{equation*}\label{eq_rep_gcpf}
		\psi(x)=\begin{cases}
			\max \big\{ \langle w_k^*, x \rangle + \gamma_k \mid k=1,\dots,r  \big\} &\text{if } x \in {\rm dom}\psi,\\
			+\infty & \text{if } x \notin {\rm dom}\psi.
		\end{cases}
	\end{equation*}	  
\end{Lemma}

\subsection{Generalized polyhedral DC optimization}\label{subsect_pdco}

Let $g, h:X\to\overline{\mathbb{R}}$ be lower semicontinuous, proper, convex functions and $C\subset X$ a nonempty closed convex set. \textit{It is assumed that} $({\rm dom}\, g)\cap C\neq\emptyset$. The minimization problem
\begin{equation}\label{DC_constraint_prob}
	{\rm Minimize}\ \; f(x):=g(x)-h(x),\ \, x\in C
\end{equation} is called a \textit{DC optimization problem} (a \textit{DC problem} for brevity). As usual, $g$ and $h$ are called the \textit{convex components} of the objective function $f$. Denote the \textit{solution set} of~\eqref{DC_constraint_prob} by ${\mathcal S}$ and the \textit{local solution set} by ${\mathcal S}_1$. By definition, $\bar x\in {\mathcal S}$ if and only if $\bar x\in C$ and $f(\bar x)\leq f(x)$ for all $x\in C$. Similarly, $\bar x\in {\mathcal S}_1$ if and only if $\bar x\in C$ and there exists a neighborhood $U$ of $\bar x$ such that $f(\bar x)\leq f(x)$ for all $x\in C\cap U$. 

Clearly, $\bar x$ is a solution (resp., a local solution) of~\eqref{DC_constraint_prob} if and only if it a solution (resp., a local solution) the unconstrained DC problem
\begin{equation}\label{DC_unconstraint_prob}
	{\rm Minimize}\ \; \big(g(x)+\delta_C(x)\big)-h(x), \ x\in X, 
\end{equation} where $\delta_C(x)= 0$ for $x\in C$, and $\delta_C(x)=+\infty$ for $x\notin C$,  is the \textit{indicator function} of the constraint set $C$.

\begin{Definition}\label{def_gpdc} {\rm If $C$ is a generalized polyhedral convex set and $g,h$ are generalized polyhedral convex functions, then~\eqref{DC_constraint_prob} is said to be a \textit{generalized polyhedral DC optimization problem} (a \textit{gpdc program} for brevity).}
\end{Definition}

According to Lemma~\ref{lem_rep_gpcf}, if~\eqref{DC_constraint_prob} is a gpdc program (resp., a pdc program), then ${\rm dom}g$ and ${\rm dom}h$ are generalized polyhedral convex sets (resp., polyhedral convex sets). In addition,  there exist vectors $u_i^*$ and $v_j^*$ in  $X^*$, real numbers $\alpha_i$ and $\beta_j$ for $i\in I=\{1,\dots,m\}$, $j\in J=\{1,\dots,q\}$, such that 
\begin{equation}\label{g}
	g(x)=\max\limits_{i\in I} \big[\langle u_i^*, x \rangle + \alpha_i]\ \; {\rm for\ all}\ x \in {\rm dom}g,
\end{equation}
\begin{equation}\label{h}
	h(x)=\max\limits_{j\in J} \big[\langle v_j^*, x \rangle + \beta_j]\ \; {\rm for\ all}\ x \in {\rm dom}h.
\end{equation}
For our convenience, we put $g_i(x)=\langle u_i^*, x \rangle + \alpha_i$, $h_j(x)=\langle v_j^*, x \rangle + \beta_j$,
\begin{equation}\label{active_indexes}
	I(x)=\left\{i\in I\;\big|\;  g_i(x) = g(x)\right\},\ \; J(x)=\left\{j\in J\; \big|\; h_j(x) = h(x)\right\}
\end{equation} for all $x\in X$.
 
\subsection{Conjugate functions and duality}
	Let $\Gamma_0 (X)$ be the set of all extended-real-valued lower semicontinuous, proper, convex functions on $X$. Clearly, any proper generalized polyhedral convex function on $X$ belongs to $\Gamma_0 (X)$. 
	
	The Fenchel \textit{conjugate function} $g^*:X^*\to\overline{\mathbb{R}}$ of a function $g:X\to\overline{\mathbb{R}}$ belonging to $\Gamma_0 (X)$ is defined by
	$$ g^*(x^*)=\sup \{\langle x^*,x\rangle-g(x) \mid x\in X\}\quad\, \forall\,x^*\in X^*.$$
	It is well known~\cite[Propostion~3, p.~174]{Ioffe_Tihomirov_1979} that $g^*: X\to \overline{\mathbb R}$  is also a lower semicontinuous, proper, convex function, i.e., $g^*\in\Gamma_0 (X^*)$. From the definition it follows that
	\begin{equation*}\label{inequality_of_conjugate_function}
		g(x)+g^*(x^*)\ge \langle x^*,x\rangle\quad \forall  x\in X,\, \forall x^*\in X^*.
	\end{equation*}
	
	Denote by $g^{**}$ the conjugate function of $g^*$, that is,
	$$ g^{**}(x)=\sup \{\langle x^*,x\rangle-g^*(x^*)\mid x^*\in X^*\} \quad \forall  x\in X.$$
	For every $g\in\Gamma_0 (X)$, by the Fenchel-Moreau theorem (see, e.g.,~\cite[Theorem~1, p.~175]{Ioffe_Tihomirov_1979}) one has $g^{**}(x)=g(x)$ for all $x\in X$.
	Many duality theorems for convex optimization problems, as well as for DC optimization problems, rely on this result.
	
The following properties of conjugate functions are well known. The proofs given in~\cite{LTY_JOGO2023} for a Hilbert space setting also apply to convex functions on locally convex Hausdorff topological vector spaces.
	
		\begin{Proposition}\label{inclusion_of_x_and_equality}{\rm (see~\cite[Proposition~1, p.~197]{Ioffe_Tihomirov_1979} and~\cite[Proposition 2.1]{LTY_JOGO2023})} For any function $g\in\Gamma_0 (X)$ and for any pair $(x,x^*)\in X\times X^*$, the inclusion $x\in \partial g^*(x^*)$ is equivalent to the equality $g(x)+g^*(x^*)=\langle x^*,x\rangle.$
	\end{Proposition}

\begin{Proposition}\label{Equivalence_of_two_inclusions} {\rm (see, e.g.,~\cite[Proposition 2.2]{LTY_JOGO2023}))} For any $g\in\Gamma_0 (X)$ and for any pair $(x,x^*)\in X\times X^*$, the inclusions $x^*\in \partial g(x) $ and $ x\in \partial g^*(x^*)$ are equivalent.
\end{Proposition}

\begin{Definition}\label{Dual_DC_Programming}{\rm For any $g,h \in \Gamma_0 (X)$, the DC program
			\begin{equation}\label{dual_program}
			{\rm Minimize}\ \; h^*(x^*)-g^*(x^*),\ \, x^*\in X^*
			\end{equation}   
	is said to be the \textit{dual problem} of the DC optimization problem
	\begin{equation}\label{DC_unconstrained_prob}
		{\rm Minimize}\ \; g(x)-h(x),\ \, x\in X.
	\end{equation}}
	\end{Definition} 

Since $X^*$ is equipped with the weak$^*$ topology by our assumption, the dual space of $X^*$ is $X$ (see, e.g.,~\cite[Lemma~2.1]{Luan_Yen_2023}), i.e., $X^{**}=X$. It is worthy to stress that \textit{$X$ can be considered either with the original topology or with the weak topology}. As $g^{**}(x)=g(x)$ and $h^{**}(x)=h(x)$ for all $x\in X$, the dual problem of~\eqref{dual_program} is~\eqref{DC_unconstrained_prob}.  

\begin{Theorem}\label{Toland-Singer's_thm} {\rm (Toland-Singer's duality theorem; see~\cite{Singer_1979,Singer_2006,Toland78})}
	The DC programs~\eqref{dual_program} and~\eqref{DC_unconstrained_prob} have the same optimal value.
\end{Theorem}

As the program~\eqref{DC_constraint_prob} can be transformed to the one in~\eqref{DC_unconstraint_prob}, the dual problem of~\eqref{DC_constraint_prob} in the sense of Definition~\ref{Dual_DC_Programming} is the following DC program
\begin{equation}\label{dual_program_1}
	{\rm Minimize}\ \; h^*(x^*)-(g+\delta_C)^*(x^*),\ \, x^*\in X^*.
\end{equation} 

Since both functions $h$ and $g+\delta_C$ in~\eqref{DC_unconstraint_prob}  belong to $\Gamma_0 (X)$, from Theorem~\ref{Toland-Singer's_thm} we get the next result.

\begin{Corollary}\label{Cor_duality} The optimal values of the DC programs~\eqref{DC_constraint_prob} and~\eqref{dual_program_1} are equal. 
\end{Corollary}

\section{Optimality Conditions}\label{sect_3}
\markboth{\sc v. t. huong, d. t. k. huyen, and n.~d.~yen}{\sc  v. t. huong, d. t. k. huyen, and n.~d.~yen}
\setcounter{equation}{0}

Consider the DC optimization problem~\eqref{DC_constraint_prob}. Recall that if $\varphi:X\to\overline{\mathbb{R}}$ is proper convex function, then the \textit{subdifferential} $\partial\varphi(\bar x)$ of $\varphi$ at a point $\bar x\in\mbox{\rm dom}\,\varphi$ is defined by setting
$$\partial\varphi(\bar x)=\{x^*\in X^*\mid \langle x^*,x-\bar x\rangle\leq\varphi(x)-\varphi(\bar x)\ \; \forall x\in X\}.$$
The \textit{normal cone} to a convex set $C\subset X$ at $\bar x\in C$ is given by
$$N_C(\bar x)=\{x^*\in X^*\mid \langle x^*,x-\bar x\rangle\leq 0\ \; \forall x\in C\}.$$

Arguing similarly as in the first part of the proof of Theorem~4.2 from~\cite{HangNTV_Yen_2016}, one can show that if $\bar x\in {\rm dom}g \cap {\rm dom}h\cap C$ is a local solution of~\eqref{DC_constraint_prob} (i.e., a local solution of~\eqref{DC_unconstraint_prob}), then the following \textit{necessary optimality condition} is satisfied:
\begin{equation}\label{DC_necc_optim_cond} \partial h(\bar x)\subset \partial \big(g+\delta_C\big)(\bar x). 
\end{equation} 

The set of the points $\bar x\in {\rm dom}g \cap {\rm dom}h\cap C$ satisfying~\eqref{DC_necc_optim_cond}, called the \textit{stationary points} of~\eqref{DC_constraint_prob}, is denoted by ${\mathcal S}_2$. In general, the fulfillment of~\eqref{DC_necc_optim_cond} is not enough for $\bar x$ to be a local minimizer of~\eqref{DC_constraint_prob}. This means that the inclusion ${\mathcal S}_1\subset {\mathcal S}_2$ can be strict (see, e.g.,~\cite[Example~2.8]{LTY_JOGO2023}).

Following Pham Dinh and Le Thi~\cite{Tao_An_AMV97,Tao_An_SIOPT98}, we call a vector $\bar x\in {\rm dom}g \cap {\rm dom}h\cap C$ a \textit{critical point} of~\eqref{DC_constraint_prob} if 
\begin{equation}\label{critical_point} \partial h(\bar x)\cap \big[\partial \big(g+\delta_C\big)(\bar x)\big]\neq\emptyset. 
\end{equation} The set of all the critical points of~\eqref{DC_constraint_prob} is denoted by ${\mathcal S}_3$.

\begin{Remark}\label{Rem1}
	{\rm If $h:X\to\overline{\mathbb{R}}$ is a proper generalized polyhedral convex function, then ${\mathcal S}_2\subset {\mathcal S}_3$. Indeed, for any point $\bar x\in {\mathcal S}_2$, we have $\bar x\in {\rm dom}\, h$. So, from~\cite[Theorems~4.10 and~4.11]{Luan_Yao_JOGO_2019} it follows that $\partial h(\bar x)$ is a nonempty generalized polyhedral convex set. Therefore,~\eqref{DC_necc_optim_cond} implies~\eqref{critical_point}. Thus, we have ${\mathcal S}_2\subset {\mathcal S}_3$.}
\end{Remark}

\begin{Remark}\label{Rem2}
	{\rm The inclusion ${\mathcal S}_2\subset {\mathcal S}_3$ can be strict even if $h:X\to\overline{\mathbb{R}}$ is a proper generalized polyhedral convex function. To justify this claim, it suffices to choose $X=\mathbb R$, $C=X$, $g(x)=0$, $h(x)=|x|$, and $\bar x=0$. Then, one has $\bar x\in {\mathcal S}_3\setminus {\mathcal S}_2$.}
\end{Remark}

The next result extends Theorem~4.2 from \cite{HangNTV_Yen_2016}, which was obtained for DC programs in finite-dimensional Euclidean spaces, to DC programs in infinite-dimensional locally convex Hausdorff topological vector spaces.  

\begin{Theorem}\label{Thm_DC_optim_cond} Let $g:X\to\overline{\mathbb {R}}$ be a proper and convex function, $h:X\to\overline{\mathbb{R}}$ a proper generalized polyhedral convex function, and $C\subset X$ a nonempty convex set. Then, a point	
\begin{equation}\label{bar_x_1st}	
	\bar x\in {\rm dom}g \cap [{\rm int}({\rm dom}h)]\cap C
\end{equation} 
is a local solution of~\eqref{DC_constraint_prob} if and only if the inclusion~\eqref{DC_necc_optim_cond} holds. Thus, the equality ${\mathcal S}_1={\mathcal S}_2$ is valid.
\end{Theorem}
\begin{proof} Fix any $\bar x\in {\rm dom}g \cap [{\rm int}({\rm dom}h)]\cap C$. If $\bar x$ is a local solution of~\eqref{DC_constraint_prob} then, as shown above,~\eqref{DC_necc_optim_cond} is valid. Conversely, suppose that~\eqref{DC_necc_optim_cond} is satisfied. Since $h$ is a generalized polyhedral convex function, it admits a representation of the form~\eqref{h}. Hence, applying the formula for computing the subdifferential of the maximum of convex functions in~\cite[Theorem 3.59]{Mordukhovich_Nam_ConvexAnalysis_2022}, we have
\begin{equation}\label{subdiff_maximum}\partial h(x)={\rm co}\left(\bigcup\limits_{j\in J(x)}\partial h_j(x)\right)={\rm co}\left\{v_j^*\;\big |\; j\in J(x)\right\}\end{equation} 
for any $x\in {\rm int}({\rm dom}h)$, where $J(x)$ is defined by~\eqref{active_indexes}. In particular, it holds that \begin{equation}\label{subdiff_maximum_bar_x}\partial h(\bar x)={\rm co}\left\{v_j^*\; \big |\; j\in J(\bar x)\right\}.\end{equation}
Put $J^{c}(\bar x)=J\setminus J(\bar x)$. Then, $h_j(\bar x)<h(\bar x)$ for any $j\in J^{c}(\bar x)$. As the functions $h_j$, $j\in J^{c}(\bar x)$, and $h$ are continuous in a neighborhood of $\bar x$, there exists a neighborhood $U\subset {\rm int}({\rm dom}h)$ of $\bar x$, such that
$h_j(x)<h(x)$ for all $x\in U$ and  $j\in J^{c}(\bar x)$. Then we get $J(x)\subset J(\bar x)$ for every $x\in U$. So, the formulas~\eqref{subdiff_maximum} and~\eqref{subdiff_maximum_bar_x} imply that $\partial h(x) \subset \partial h(\bar x)$ for every $x\in U$. 

Since $J(x)\neq\emptyset$ for any $x\in {\rm dom}h$, from~\eqref{subdiff_maximum} we can deduce that $\partial h(x)\neq\emptyset$ for all $x\in U$. So, for any $x\in U$, combing the inclusion $\partial h(x) \subset \partial h(\bar x)$ with ~\eqref{DC_necc_optim_cond} yields a subgradient $x^*\in \partial h(x)$ such that $x^*\in \partial \big(g+\delta_C\big)(\bar x)$. Therefore, we have
\begin{eqnarray*}
\begin{cases}
	h(y) - h(x)\geq \langle x^*,\, y - x\rangle,\\
	g(y)+\delta_C(y) - g(\bar x)\geq \langle x^*,\, y - \bar x\rangle
\end{cases}	
\end{eqnarray*}
for all $y\in X$. Substituting $y= \bar x$ to the first inequality and 
$y= x$ to the second one, we obtain
\begin{eqnarray*}
\begin{cases}
	h(\bar x) - h(x)\geq \langle x^*,\, \bar x - x\rangle\\
	g(x)+\delta_C(x) - g(\bar x)\geq \langle x^*,\,  x - \bar x\rangle.
\end{cases}	
\end{eqnarray*} 
Adding the last inequalities side by side yields
\begin{equation*}
	[g(x)+\delta_C(x)-h(x)]-[g(\bar x)-h(\bar x)]\geq 0.
\end{equation*} It follows that 
\begin{equation*}
	[g(x)-h(x)]-[g(\bar x)-h(\bar x)]\geq 0\quad \forall x\in C\cap U.
\end{equation*} 
Thus, $\bar x$ is a local solution of~\eqref{DC_constraint_prob}. $\hfill\Box$
\end{proof}

\medskip
Specializing Theorem~\ref{Thm_DC_optim_cond} to the case where both $h$ and $g$ are generalized polyhedral convex functions, we have the following theorem.

\begin{Theorem}\label{Thm_DC_necc_optim_cond_1} Let the functions $g:X\to\overline{\mathbb {R}}$ and $h:X\to\overline{\mathbb{R}}$ be generalized polyhedral convex functions given by~\eqref{g} and~\eqref{h}, and $C\subset X$ a nonempty convex set. Then, a point	\begin{equation}\label{bar_x} \bar x\in [{\rm int}({\rm dom}g)]\cap [{\rm int}({\rm dom}h)]\cap C
\end{equation} 
is a local solution of~\eqref{DC_constraint_prob} if and only if 
\begin{equation}\label{DC_necc_optim_cond_1} {\rm co}\left\{v_j^*\; \big |\; j\in J(\bar x)\right\}\subset  {\rm co}\left\{u_i^*\; \big |\; i\in I(\bar x)\right\}+N_C(\bar x), 
\end{equation} where the active index sets $I(\bar x)$ and $J(\bar x)$ are defined by~\eqref{active_indexes}.
\end{Theorem}
\begin{proof} Let $\bar x$ be such that~\eqref{bar_x} is satisfied. Then, by Theorem~\ref{Thm_DC_optim_cond}, $\bar x$ is a local solution of~\eqref{DC_constraint_prob} if and only if \eqref{DC_necc_optim_cond} is fulfilled. On one hand, since $g$ and $h$ admit the representations~\eqref{g} and~\eqref{h}, $\bar x\in {\rm int}({\rm dom}g)$, and $\bar x\in {\rm int}({\rm dom}h)$, we have $$\partial g(\bar x)={\rm co}\left\{u_i^*\; \big |\; i\in I(\bar x)\right\},\ \; \partial h(\bar x)={\rm co}\left\{v_j^*\; \big |\; j\in J(\bar x)\right\}.$$ On the other hand, as $\bar x\in {\rm dom}\delta_C$ and $g$ is continuous at $\bar x$, 
applying the Moreau-Rockafellar theorem for subdifferentials of convex functions on locally convex Hausdorff topological vector spaces (see, e.g.,~\cite[Theorem~0.3.3, p.~47]{Ioffe_Tihomirov_1979}) yields
$$\partial \big(g+\delta_C\big)(x)=\partial g(x)+\partial\delta_C(x)=\partial g(x)+N_C(x)$$ for all $x\in X$. In particular,
$$\partial \big(g+\delta_C\big)(\bar x)=\partial g(\bar x)+\partial\delta_C(\bar x)=\partial g(\bar x)+N_C(\bar x).$$
So, the condition~\eqref{DC_necc_optim_cond} is equivalent to~\eqref{DC_necc_optim_cond_1}. Thus, we have proved that a point $\bar x$ satisfying~\eqref{bar_x} is a local solution of~\eqref{DC_constraint_prob} if and only if~\eqref{DC_necc_optim_cond_1} holds. $\hfill\Box$
\end{proof}

\medskip
Now, in addition to the assumptions of Theorem~\ref{Thm_DC_necc_optim_cond_1}, if $C$ is a generalized polyhedral convex set, then we have the next result.

	\begin{Theorem}\label{Thm_DC_optim_cond_2} Suppose that $g:X\to\overline{\mathbb {R}}$ and $h:X\to\overline{\mathbb{R}}$ are generalized polyhedral convex functions given respectively by~\eqref{g} and~\eqref{h}. If $C\subset X$ is a nonempty generalized polyhedral convex set defined by~\eqref{equivalence_gpcs}, then a point $\bar x$ satisfying~\eqref{bar_x} is a local solution of~\eqref{DC_constraint_prob} if and only if 
		\begin{equation}\label{DC_necc_optim_cond_2} {\rm co}\left\{v_j^*\; \big |\; j\in J(\bar x)\right\}\subset  {\rm co}\left\{u_i^*\; \big |\; i\in I(\bar x)\right\}+{\rm pos}\big\{x^*_k \mid k\in K(\bar x)\big\}+({\rm ker}\,A)^{\perp},
		\end{equation} where $K(\bar x)$ consists of the indexes  $k\in\{1,\ldots, p\}$ such that $\langle x^*_k,\bar x \rangle=\alpha_k$,
	$${\rm pos}\left\{x^*_k \mid k\in K(\bar x)\right\}:=\left\{\sum\limits_{k\in K(\bar x)}\lambda_k x^*_k \mid \lambda_k\geq 0\ {\rm  for\ all}\ k\in K(\bar x)\right\},$$ ${\rm ker}\,A=\{x\in X\mid Ax=0\}$ and $({\rm ker}\,A)^{\perp}=\{x^*\in X^*\mid \langle x^*,x\rangle =0 \ {\rm  for\ all}\ x\in {\rm ker}\,A\}$.
	\end{Theorem}
	\begin{proof} To specialize~\eqref{DC_necc_optim_cond_1} for the case where $C$ is of the form~\eqref{equivalence_gpcs}, we apply the formula of the normal cone in~\cite[Proposition~4.2]{Luan_Yao_Yen_2018} to get
		\begin{equation}\label{nornal_cone_gpcs}
			N_{C}(\bar x)={\rm pos}\big\{x^*_k \mid k\in K(\bar x)\big\}+({\rm ker}A)^{\perp}.
		\end{equation} By~\eqref{nornal_cone_gpcs} we can assert that, under the assumption made on $C$, the condition~\eqref{DC_necc_optim_cond_1} is equivalent to~\eqref{DC_necc_optim_cond_2}. So, the desired result follows from Theorem~\ref{Thm_DC_necc_optim_cond_1}. $\hfill\Box$
	\end{proof}

\begin{Example}\label{Example1}
	{\rm Consider the DC optimization problem~\eqref{DC_constraint_prob} with $X=\mathbb R$, $C=[-2,3]$, $g(x)=0$ for all $x\in\mathbb R$, and 
	$$h(x)=\begin{cases} -x-1 & {\rm for}\ x<-1\\
		0 & {\rm for}\ x\in [-1,1]\\
		x-1 & {\rm for}\ x>1.	
\end{cases}$$ Since $$\partial(g+\delta_C)(x)=\partial\delta_C(x)=\begin{cases} (-\infty,0] & {\rm for}\ x=-2\\
\{0\} & {\rm for}\ x\in (-2,3)\\
[0,+\infty) & {\rm for}\ x=3\\
\emptyset & {\rm otherwise}\end{cases}$$ and 
$$\partial h (x)=\begin{cases} \{-1\} & {\rm for}\ x<-1\\
	[-1,0] & {\rm for}\ x=-1\\
	\{0\}  & {\rm for}\ x\in (-1,1)\\
	[0,1] & {\rm for}\ x=1\\
	\{1\} & {\rm for}\ x>1,\end{cases}$$ the inclusion~\eqref{DC_necc_optim_cond} holds for $\bar x$ if and only if $\bar x$ belongs to the set $\{-2\}\cup (-1,1)\cup \{3\}$. Therefore, by Theorem~\ref{Thm_DC_optim_cond} we have
\begin{equation}\label{S1_2} {\mathcal S}_1={\mathcal S}_2=\{-2\}\cup (-1,1)\cup \{3\}.
\end{equation} As $C$ is compact and $g-h$ is continuous on $C$,~\eqref{DC_constraint_prob} have a solution. Comparing the values of the objective function at the points of ${\mathcal S}_1$ yields ${\mathcal S}=\{3\}$. It is easy to see that~\eqref{critical_point} is satisfied if and only if $\bar x$ belongs to the set $\{-2\}\cup [-1,1]\cup \{3\}$. So, we have 
${\mathcal S}_3=\{-2\}\cup [-1,1]\cup \{3\}.$}
\end{Example}
\begin{figure}[htp]
	\begin{center}
		\includegraphics[height=8.5cm,width=12cm]{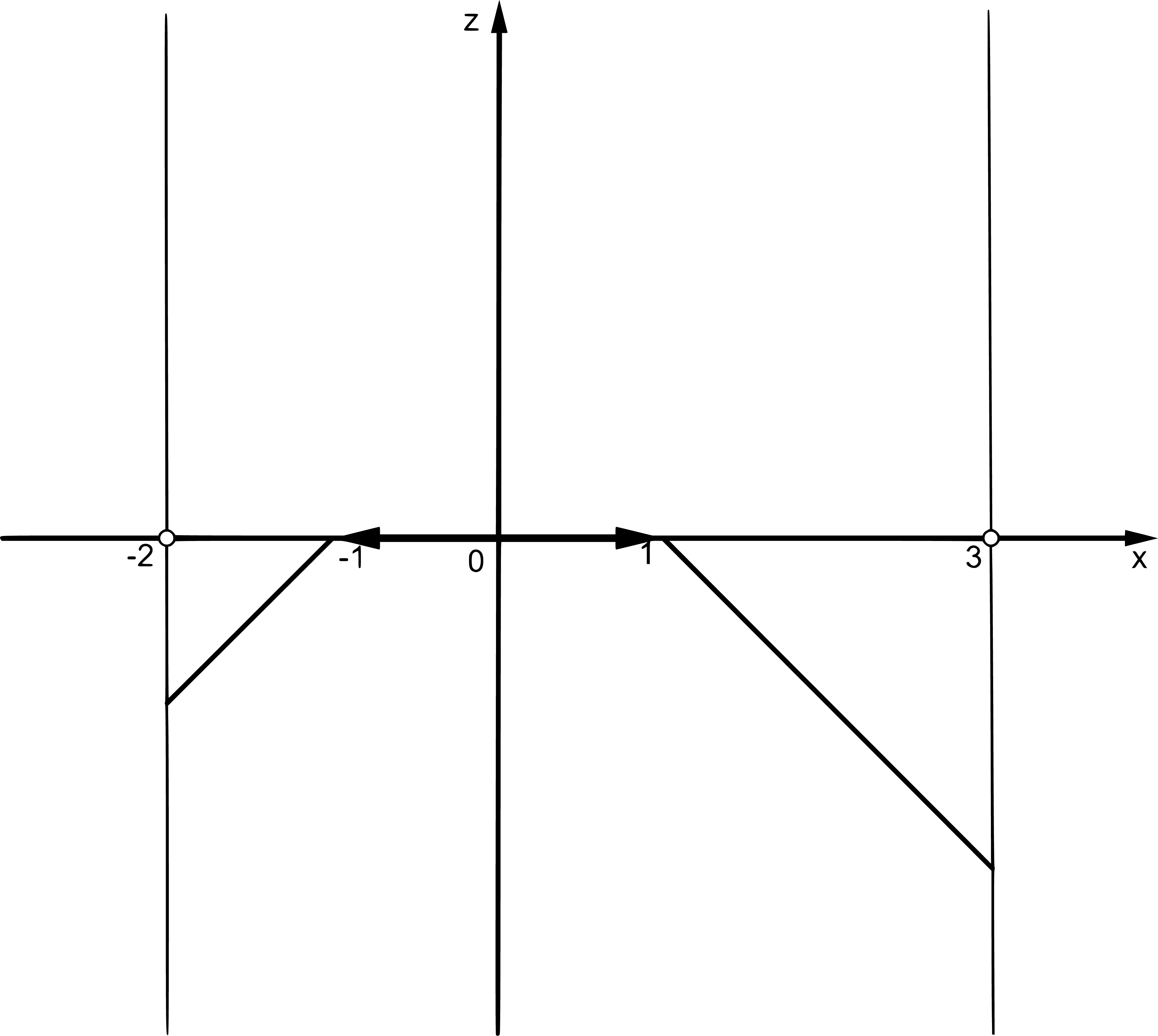}
	\end{center}
	\caption{The sets $\mathcal{S}$, $\mathcal{S}_1$, $\mathcal{S}_2$, and $\mathcal{S}_3$ in Example~\ref{Example1}}\label{Fig_1}
\end{figure}

\section{Local Solution Sets and Global Solution Sets}\label{sect_4}
\markboth{\sc v. t. huong, d. t. k. huyen, and n.~d.~yen}{\sc  v. t. huong, d. t. k. huyen, and n.~d.~yen}
\setcounter{equation}{0}

Based on Theorem~\ref{Thm_DC_optim_cond}, we can obtain our first result on the structure of the local solution set of the DC program~\eqref{DC_constraint_prob}. 

\begin{Theorem}\label{thm_structure1} Let $g:X\to\overline{\mathbb {R}}$ be a lower semicontinuous, proper, convex function, $h:X\to\overline{\mathbb{R}}$ a proper generalized polyhedral convex function, and $C\subset X$ a nonempty closed convex set. If both sets ${\rm dom}g$ and ${\rm int}({\rm dom}h)$ contain $C$, then the local solution set ${\mathcal S}_1$ of~\eqref{DC_constraint_prob} is the union of finitely many convex subsets of $C$, while the global solution set ${\mathcal S}$ of the problem is the union of finitely many closed convex subsets of $C$. 
\end{Theorem}
\begin{proof} Suppose that $h$ is given by~\eqref{h}. Since ${\rm dom}g\supset C$ and  $[{\rm int}({\rm dom}h)]\supset C$ by our assumptions, we have ${\rm dom}g \cap [{\rm int}({\rm dom}h)]\cap C=C$. Therefore, condition~\eqref{bar_x_1st} is satisfied for every point $\bar x\in C$. So, by Theorem~\ref{Thm_DC_optim_cond}, a point $\bar x\in C$ belongs to ${\mathcal S}_1$ if and only if $\bar x\in {\mathcal S}_2$, i.e., the inclusion~\eqref{DC_necc_optim_cond} holds. Since every point of $C$ belongs to ${\rm int}({\rm dom}h)$, the formula for computing the subdifferential of the maximum of convex functions in~\cite[Theorem~3.59]{Mordukhovich_Nam_ConvexAnalysis_2022} gives \begin{equation}\label{subdiff_h} \partial h(\bar x)={\rm co}\left\{v_j^*\; \big |\; j\in J(\bar x)\right\}\quad \forall\bar x\in C.
\end{equation}  	
By~\eqref{subdiff_h} and the convexity  the set on the right-hand side of~\eqref{DC_necc_optim_cond},  we can rewrite the latter equivalently as
\begin{equation*}\label{v*j} v_j^*\in \partial \big(g+\delta_C\big)(\bar x)\quad \forall j\in J(\bar x). 
\end{equation*}

{\sc Claim 1}. \textit{Suppose that $x\in C$. Then $x$ belongs ${\mathcal S}_1$ if and only if there exists a nonempty subset $J_1\subset J$ such that 
\begin{equation}\label{incl_new} J(x)\subset J_1
\end{equation}
and
\begin{equation}\label{v*j} v_j^*\in \partial \big(g+\delta_C\big)(x)\quad \forall j\in J_1. 
\end{equation}}

Indeed, if $x\in {\mathcal S}_1$, then by choosing $J_1=J(x)$ and using the above-recalled optimality condition we have~\eqref{incl_new} and~\eqref{v*j}. Conversely, if~\eqref{incl_new} and~\eqref{v*j} are fulfilled, then we have 
$${\rm co}\left\{v_j^*\; \big |\; j\in J(x)\right\}\subset \partial \big(g+\delta_C\big)(x).$$ Combining this with~\eqref{subdiff_h} yields 
$$ \partial h(x)\subset \partial \big(g+\delta_C\big)(x).$$ So, we have $x\in {\mathcal S}_2$. Now, applying Theorem~\ref{Thm_DC_optim_cond}, we get $x\in {\mathcal S}_1$.

{\sc Claim 2}. \textit{For any nonempty subset $J_1\subset J$, the set $C_{J_1}$ of all $x\in C$ satisfying~\eqref{incl_new} is a convex subset of $C$, which is open in the relative topology (see~{\rm \cite[p.~51]{Kelley_1975}}) of $C$.}

Indeed, if $C_{J_1}=\emptyset$, then we are done. Now, suppose that $C_{J_1}\neq\emptyset$ and let $x\in C_{J_1}$ be given arbitrarily. As it has been shown in the first part of the proof of Theorem~\ref{Thm_DC_optim_cond}, by~\eqref{incl_new} we can find a neighborhood $U\subset {\rm int}({\rm dom}h)$ of $x$, such that
$h_j(x')<h(x')$ for all $x'\in U$ and $j\in J^{c}(x):=J\setminus J(x)$. This implies that $J(x')\subset J(x)\subset J_1$ for every $x'\in U$. Therefore, $U\cap C\subset C_{J_1}$. We have thus shown that $C_{J_1}$ is open in the relative topology of $C$. To prove the convexity of $C_{J_1}$, take any $x^1, x^2\in C_{J_1}$ and $t\in (0,1)$. Since $J\big(x^1\big)\subset J_1$ and $J\big(x^2\big)\subset J_1$, for every $j'\in J\setminus J_1$ one has $j'\notin J\big(x^1\big)$ and $j'\notin J\big(x^2\big)$. So, the strict inequalities
$$\langle v_{j'}^*, x^1 \rangle + \beta_{j'}<\max\limits_{j\in J_1} \big[\langle v_j^*,x^1 \rangle + \beta_j]$$
and
$$\big[\langle v_{j'}^*, x^2 \rangle + \beta_{j'}<\max\limits_{j\in J_1} \big[\langle v_j^*, x^2 \rangle + \beta_j]$$ hold. Multiplying the first inequality by $(1-t)$, the second one by $t$, and adding the resulted inequalities yield
$$\langle v_{j'}^*,(1-t)x^1+tx^2\rangle + \beta_{j'}<\max\limits_{j\in J_1} \big[\langle v_j^*, (1-t)x^1+tx^2\ \rangle + \beta_j].$$ It follows that $j'\notin J\big((1-t)x^1+tx^2\big)$. As this is valid for every $j'\in J\setminus J_1$, we must have $J\big((1-t)x^1+tx^2\big)\subset J_1$.
Therefore, $(1-t)x^1+tx^2\in C_{J_1}$.

The inclusion in~\eqref{v*j} can be rewritten equivalently as 
\begin{equation}\label{inclusion} 0\in \partial \big(g+\delta_C-v_j^*\big)(x).\end{equation}
Since the function $\varphi_j:X\to\overline{\mathbb {R}}$ defined by $$\varphi_j(x):=g(x)+\delta_C(x)-v_j^*(x),\quad x\in X,$$ is convex, 
by the a generalized version of the Fermat stationary rule
for convex extended-real-valued functions in~\cite[Proposition~3.29]{Mordukhovich_Nam_ConvexAnalysis_2022} we know that the set of $x\in X$ satisfying the inclusion~\eqref{inclusion} coincides with the solution set of the convex optimization problem
\begin{equation*}\label{optim_j}
	\min\{\varphi_j(x)\mid x\in X\},
\end{equation*} 
which is denoted by $\Omega_j$. As $\varphi_j$ is a lower semicontinuous function, $\Omega_j$ is convex and closed. Noting that the set of $x\in X$ satisfying~\eqref{v*j} is $\bigcap\limits_{j\in J_1}\Omega_j$ and the set $C_{J_1}$ of all $x\in C$ satisfying~\eqref{incl_new} is convex and open in the relative topology of $C$ by Claim~2, we can assert that the set of $x\in X$ fulfilling both conditions~\eqref{incl_new} and~\eqref{v*j}, denoted by ${\mathcal S}_1(J_1)$, is a convex subset of $C$. By Claim~1 we have
\begin{equation}\label{union}{\mathcal S}_1=\bigcup\Big\{{\mathcal S}_1(J_1)\mid J_1\subset J,\ J_1\neq\emptyset\Big\}.\end{equation}
Hence, the local solution set of~\eqref{DC_constraint_prob} is the union of finitely many convex subsets of~$C$. 

To prove the second assertion of the theorem, we can argue similarly as in~\cite{Tao_An_AMV97}. Namely, by~\eqref{h} and the assumption ${\rm int}({\rm dom}h)\supset C$ we have
$$\begin{array}{rl}
	\inf\limits_{x\in X} \left[\big(g(x)+\delta_C(x)\big)-h(x)\right] &= \inf\limits_{x\in X} \left[\big(g(x)+\delta_C(x)\big)-\max\limits_{j\in J} \big(\langle v_j^*, x \rangle + \beta_j\big)\right]\\ 
	&= \inf\limits_{x\in X} \left[\big(g(x)+\delta_C(x)\big)+\min\limits_{j\in J}	 \big(-\langle v_j^*, x \rangle - \beta_j\big)\right]\\
	&= \inf\limits_{x\in X}\min\limits_{j\in J}\left[\big(g(x)+\delta_C(x)\big)-\langle v_j^*, x \rangle - \beta_j\right]\\ 
	&= \min\limits_{j\in J}\inf\limits_{x\in X} \left[\big(g(x)+\delta_C(x)\big)-\langle v_j^*, x \rangle - \beta_j\right].
\end{array}$$ Therefore, denoting by $\alpha_j$ (resp., $\mathcal{S}^j$) the optimal value (resp., the solution set) of the convex optimization problem
\begin{equation}\label{optim_j}	{\rm Minimize}\ \, \big(g(x)+\delta_C(x)\big)-\langle v_j^*, x \rangle - \beta_j,\ \; x\in X,
\end{equation} and letting $\bar\alpha$ stand for the optimal value of~\eqref{DC_unconstraint_prob}, we get $\bar\alpha=\min\limits_{j\in J}\alpha_j$. Hence, the representation ${\mathcal S}=\bigcup\limits_{j\in J_*}\mathcal{S}^j$, where $J_*:=\{j\in J\mid \alpha_j=\bar\alpha\}$, is valid. Clearly, the closedness of $C$, the lower semicontinuity of $g$ and $h$, and the assumptions ${\rm dom}g\supset C$ and ${\rm int}({\rm dom}h)\supset C$ imply that $\mathcal{S}^j$ is a closed subset of $C$ for each $j\in J$. So, we see that ${\mathcal S}$ is the union of finitely many closed convex subsets of $C$.  
$\hfill\Box$
\end{proof}

\medskip We call the intersection of a generalized polydedral convex subset of $C$ and convex subset of $C$, which is open in the relative topology of $C$, a \textit{semi-closed generalized polydedral convex subset} of $C$.

\medskip
Our second result on the structure of the local solution set of the DC program~\eqref{DC_constraint_prob} is stated as follows.

\begin{Theorem}\label{thm_structure2} Let $g,h:X\to\overline{\mathbb {R}}$ be generalized polyhedral convex functions and let $C\subset X$ be a nonempty generalized polyhedral convex set. If the conditions ${\rm dom}g\supset C$ and~${\rm int}({\rm dom}h)\supset C$ are satisfied, then the local solution set of~\eqref{DC_constraint_prob}, now denoted by~${\mathcal S}_{1a}$,  is the union of finitely many semi-closed generalized polydedral convex subsets of $C$, while and global solution set~${\mathcal S}$ of the problem is the union of finitely many generalized polyhedral convex sets. 
\end{Theorem}
\begin{proof} By the conditions made on $g$, $h$,and $C$, all the assumptions of Theorem~\ref{thm_structure1} are satisfied. Therefore, since the local solution set ${\mathcal S}_{1a}$ is a special case of the local solution set ${\mathcal S}_1$ of ~\eqref{DC_constraint_prob}, we can repeat the preceding proof and get from~\eqref{union} that \begin{equation}\label{union2}{\mathcal S}_{1a}=\bigcup\Big\{{\mathcal S}_1(J_1)\mid J_1\subset J,\ J_1\neq\emptyset\Big\}\end{equation}
with ${\mathcal S}_1(J_1)$ being the set of $x\in X$ fulfilling both conditions~\eqref{incl_new} and~\eqref{v*j}.

Let $\varphi_j$ and $\Omega_j$ be defined as in the proof of Theorem~\ref{thm_structure1}. By our assumptions, the functions $g$, $\delta_C$, and $(-v_j^*)$ are generalized polyhedral convex, $${\rm dom}g\cap {\rm dom}\delta_C\cap {\rm dom}(-v_j^*)= C$$ and $C$ is nonempty, applying Theorem~3.7 from~\cite{Luan_Yao_Yen_2018} implies that $\varphi_j$ is a generalized polyhedral convex function. Therefore, by~\cite[Proposition~3.9]{Luan_Yao_JOGO_2019}, $\Omega_j$ is a  generalized polyhedral convex set. It follows that $\bigcap\limits_{j\in J_1}\Omega_j$ is a generalized polyhedral convex set. Since the set of $x\in X$ satisfying~\eqref{v*j} is $\bigcap\limits_{j\in J_1}\Omega_j$ and the set $C_{J_1}$ of all $x\in C$ satisfying~\eqref{incl_new} is convex and open in the relative topology of $C$ (see Claim~2 in the above proof), we have shown that ${\mathcal S}_1(J_1)$ is a semi-closed generalized polyhedral convex subset of $C$ for each nonempty subset $J_1\subset J$. Thus, the first assertion of the theorem follows from~\eqref{union2}.

To prove the second assertion of the theorem, it suffices to reapply the arguments of the second part of the proof of Theorem~\ref{thm_structure1} and observe that our assumptions  assure that, for each $j\in J$, the optimization problem in~\eqref{optim_j} is generalized polyhedral convex; hence~$\mathcal{S}^j$ is a generalized polyhedral convex set by~\cite[Proposition~3.9]{Luan_Yao_JOGO_2019}. $\hfill\Box$
\end{proof}

 Recall (see~\cite[p.~53]{Kelley_1975} for an equivalent definition) that a topological space $Z$ is called \textit{connected} if it cannot be represented as the union of two disjoint nonempty open sets. Since the complement of an open set is a closed set, a topological space is connected if and only if it cannot be represented as the union of two disjoint nonempty closed sets. A subset $Y$ of $Z$ is said to be connected if the topological space $Y$ with the \textit{relative topology} (see~\cite[p.~51]{Kelley_1975}) is connected. Thus, $Y$ is connected if and only if there are no open sets $U$ and $V$ of $Z$ such that $U\cap Y\neq\emptyset$, $V\cap Y\neq\emptyset$, and $Y=(U\cap Y)\cup (V\cap Y)$. A \textit{connected component} (called a \textit{component} in~\cite[p.~54]{Kelley_1975}) of a topological space is a maximal connected subset; that is, a connected subset which is properly contained in no other  connected subset.

\medskip
To proceed furthermore, we need the following useful lemmas, which might be new. 

\begin{Lemma}\label{Lem2}
	Suppose that $D$ is a subset of a topological vector space $Z$ and $C$ is a convex subset of $D$. Then, the closure of $C$ in the relative topology of $D$, denoted by $\widehat C$, is convex.
\end{Lemma}
\begin{proof}
Let $z^1$ and $z^2$ be arbitrary points of  $\widehat C$. Given any $t\in (0,1)$, we have to show that $z:=(1-t)z^1+tz^2$ belongs to $\widehat C$. Let $U\subset Z$ be any open set containing~$z$. The map $\psi: Z\times Z\to Z$, where $\psi(x,y):=(1-t)x+ty$ for all $x,y\in Z$, is continuous. Since $\psi(z^1,z^2)=z$, there exist open sets $U_1\subset Z$ and $U_2\subset Z$ such that $z^1\in U_1$, $z^2\in U_2$, and $\psi(x,y)\in U$ for all $x\in U_1$ and $y\in U_2$. As $U_1\cap D$ is a neighborhood of $z^1$ in the relative topology of $D$, we can find $u^1\in (U_1\cap D)\cap C$. Since  $(U_1\cap D)\cap C= U_1\cap D\cap C= U_1\cap C$, we have $u^1\in U_1\cap C$. Arguing similarly, we can find $u^2\in U_2\cap C$. On one hand, by the convexity of $C$ we have $(1-t)u^1+tu^2\in C$. On the other hand, by the above construction of $U_1$ and $U_2$, 
$$\psi(u^1,u^2)=(1-t)u^1+tu^2\in U.$$
Therefore, $(1-t)u^1+tu^2\in U\cap C$. As $(U\cap D)\cap C=U\cap C$, this implies that $(U\cap D)\cap C$ is nonempty. Thus, we have proved that $z\in \widehat C$. $\hfill\Box$
\end{proof}

\begin{Lemma}\label{Lem3}
	Suppose that $D=\bigcup\limits_{k=1}^mC_k$, where $C_1,\dots,C_m$ are nonempty convex sets in a topological vector space $Z$. If $D$ is connected in the relative topology, then $D$ is connected by line segments, i.e., for any $z,w\in D$, there exists a finite sequence of line segments $L_i:=[z_i,z_{i+1}]$, $i=1,...,r-1$, such that $z_1=z$,  and $z_r=w$, and $L_i\subset D$ for $i=1,...,r-1$.
\end{Lemma}
\begin{proof} Put $K=\{1,\dots,m\}$. From the equality $D=\bigcup\limits_{k=1}^mC_k$ we deduce that 
	\begin{equation}\label{D_uninion}
		D=\bigcup\limits_{k=1}^m\widehat C_k,
	\end{equation}
where each $\widehat C_k$, $k\in K$, stands for the closure of $C_k$ in the relative topology of $D$. Indeed, since $C_k\subset \widehat C_k$ for every $k\in K$, the inclusion $D\subset\bigcup\limits_{k=1}^m\widehat C_k$ holds. The reverse inclusion is valid because $\widehat C_k\subset D$ for each $k\in K$.
	
	Pick any $k_1\in K$. If there is no $k\in K\setminus\{k_1\}$ such that $\widehat C_k\cap \widehat C_{k_1}\neq\emptyset$, then we put $A=\widehat C_{k_1}$. Since $\widehat C_{k_1}$ is convex by the convexity of $C_{k_1}$ and by Lemma~\ref{Lem2}, for any $z,w\in\widehat C_{k_1}$, the line segment $[z,w]$ lies in $\widehat C_{k_1}$. Thus,~$A$ is connected by line segments. 
	
	If there exists an index $k_2\in K\setminus\{k_1\}$ such that $\widehat C_{k_2}\cap\widehat C_{k_1}\neq\emptyset$, then the set $\Omega:=\widehat C_{k_1}\cup\widehat C_{k_2}$ is connected by line segments. Indeed, given any $z,w\in \Omega$, we see at once that there is a line segment joining $z$ with $w$ if both points $z$ and $w$ lie either in $\widehat C_{k_1}$ or in $\widehat C_{k_2}$. In the remaining situation, we may suppose that $z\in\widehat C_{k_1}$ and $w\in\widehat C_{k_2}$. Taking any $u\in \widehat C_{k_2}\cap\widehat C_{k_1}$, we have $[z,u]\subset \widehat C_{k_1}$ and $[u,w]\subset \widehat C_{k_2}$. Hence, $z$ can be joined with  $w$ by a sequence of two line segments, where each segment is contained in $\Omega$.
	
	If there is no $k\in K\setminus\{k_1,k_2\}$ such that $\widehat C_k\cap \Omega\neq\emptyset$, then we put $A=\widehat C_{k_1}\cup\widehat C_{k_2}$ and notice that~$A$ is connected by line segments. 
	
	If there exists an index $k_3\in K$ such that  $\widehat C_{k_3}\cap \Omega\neq\emptyset$, then we enlarge $\Omega$ by setting $\Omega=\widehat C_{k_1}\cup\widehat C_{k_2}\cup\widehat C_{k_3}$. Arguing similarly as above, one can easily show that any two points of~$\Omega$ can be joined by a sequence of line segments consisting of at most three segments, where each segment is contained in $\Omega$. 
	
	If $D=\Omega$, then we are done. Otherwise, we can find an index $k\in K\setminus\{k_1,k_2,k_3\}$ such that  $\widehat C_{k}\cap \Omega\neq\emptyset$, and we can enlarge $\Omega$ by letting this $\widehat C_{k}$ join the existing union of $\widehat C_{k_i}$, $i=1,2,3$. Again, the new set $\Omega$ is connected by line segments... By continuing this process, we can get a set $\Omega$ of the form $\Omega=\widehat C_{k_1}\cup\widehat C_{k_2}\cup\widehat C_{k_3}\cup\cdots\cup \widehat C_{k_s}$ with $k_i\in K$ for all $i\in\{1,\dots,s\}$, which is connected by line segments, such that there is no $k\in K\setminus\{k_1,\dots,k_s\}$ satisfying $\widehat C_{k}\cap \Omega\neq\emptyset$. Put $A=\Omega$.
	
	The proof is completed if $D=A$. Otherwise, the union of all $\widehat C_{k}$ with $k\in K$ and $\widehat C_{k}\cap A=\emptyset$, which is denoted by $B$, is nonempty. Clearly,  from our construction it follows that both sets~$A$ and~$B$ are closed in the relative topology of $D$. Since $A$ and $B$ are nonempty and disjoint,~$D$ is disconnected in the relative topology. But this contradicts the assumptions made on~$D$.  
	
	Thus, we have shown that $D$ is connected by line segments. $\hfill\Box$ \end{proof}

\begin{Theorem}\label{thm_const_value} If the assumptions of Theorem~\ref{thm_structure1} are satisfied and the restriction of $g$ on each line segment contained in $C$ is a continuous function, then every connected component of~${\mathcal S}_1$ is a union of finitely many convex sets and the function~$f:=g-h$ has a constant value on the component.
\end{Theorem}
\begin{proof} By Theorem~\ref{thm_structure1}, the local solution set ${\mathcal S}_1$ of the DC problem~\eqref{DC_constraint_prob} is a union of finitely many convex subsets of $C$. Denote these convex sets by $A_1,\ldots,A_\ell$ and put $K=\{1,\ldots,\ell\}$. Let $\Omega$ be a connected component of ${\mathcal S}_1$. Clearly, if $A_k\cap\Omega\neq\emptyset$ for some $k\in K$, then $A_k\subset\Omega$. So, $\Omega$ is the union of several sets $A_k$, $k\in K$. Therefore, without any loss of generality, we may assume that 
		\begin{equation}\label{connected_component}
		\Omega=A_1\cup \cdots \cup A_{\ell_1},
		\end{equation}
where $\ell_1\leq\ell$. The first assertion of the theorem has been proved.

Since $\Omega$ is connected and every set $A_k, k\in K$, is convex, by~\eqref{connected_component} and by Lemma~\ref{Lem3} we can infer that is connected by line segments.

Given any $z,w\in\Omega$, we find a finite sequence of line segments $L_i:=[z_i,z_{i+1}]$, $i=1,...,r-1$, such that $z_1=z$,  and $z_r=w$, and $L_i\subset \Omega$ for $i=1,...,r-1$.

To complete the proof, it suffices to show that the function $f=g-h$ is constant on each line segment $L_i$, where $i\in\{1,\ldots,r-1\}$. By our assumptions, $g:X\to\overline{\mathbb {R}}$ is a proper convex function, $h:X\to\overline{\mathbb{R}}$ is a proper generalized polyhedral convex function, and both sets ${\rm dom}g$ and ${\rm int}({\rm dom}h)$ contain~$C$. In addition, the restriction of $g$ on $L_i$ is continuous. By~\eqref{h}, the restriction of $h$ on $L_i$ admits the representation $h(x)=\max\limits_{j\in J} \big[\langle v_j^*, x \rangle + \beta_j]$. Thus, the continuity of the affine functions $x\mapsto \langle v_j^*, x \rangle + \beta_j$, $j\in J$, implies that the restriction of $h$ on $L_i$ is continuous. Therefore, the restriction of $f=g-h$ on $L_i$ is continuous. By the Weierstrass theorem (see, e.g.,~\cite[Theorem~1.59]{Mordukhovich_Nam_ConvexAnalysis_2022}), $f$ attains its maximum value $\gamma\in\mathbb R$ on $L_i$ at a point $\bar x$. Let 
	$D=\big\{x\in L_i\mid f(x)=\gamma\big\}.$ Then, $D$ is closed and nonempty. For any $u\in D$, as $u\in L_i$ and $L_i\subset {\mathcal S}_1$, $u$ is a local solution of~\eqref{DC_constraint_prob}. Hence, there is an open a neighborhood $U$ of $u$ such that $f(u)\leq f(x)$ for all $x\in C\cap U$. In particular, 
	$$f(u)\leq f(x)\quad \forall x\in L_i\cap U.$$
 On the other hand, since $f(u)=\gamma$, we have $f(u)\geq f(x)$ for all $x\in L_i.$ So, we must have $f(x)=\gamma$ for all $x\in L_i\cap U.$ This means that $D$ contains the set $L_i\cap U$. As the latter is an open set in the relative topology of $L_i$, we have proved that $D$ is an open set in the relative topology of $L_i$. Thus, $D$ is a nonempty subset of $L_i$, which is both closed and open. Then, by the connectedness of $L_i$, we can infer that $D=L_i$ (see~\cite[Proposition~1.61]{Mordukhovich_Nam_ConvexAnalysis_2022}). Thus, $f$ is constant on $L_i$.	
 
 The proof is complete. $\hfill\Box$
\end{proof}

\begin{Theorem}\label{thm_const_value_2} If the assumptions of Theorem~\ref{thm_structure2} are satisfied, then every connected component of~${\mathcal S}_{1a}$ is a union of finitely many semi-closed generalized polyhedral convex sets and the function~$f:=g-h$ has a constant value on the component.
\end{Theorem}
\begin{proof} By Theorem~\ref{thm_structure2}, the local solution set ${\mathcal S}_{1a}$ of the generalized polyhedral DC problem~\eqref{DC_constraint_prob} is a union of finitely many semi-closed generalized polyhedral convex sets. In addition, since $g:X\to\overline{\mathbb {R}}$ is a generalized polyhedral convex function and ${\rm dom}g\supset C$, the restriction of $g$ on each line segment contained in $C$ is a continuous function. So, to obtain the desired conclusions, it suffices to repeat some arguments used in the proof of Theorem~\ref{thm_const_value}. $\hfill\Box$
\end{proof}

\begin{Remark}{\rm The results in Theorem~\ref{thm_const_value} (resp., in Theorem~\ref{thm_const_value_2}) show that the local solution set~${\mathcal S}_1$ (resp., the local solution set~${\mathcal S}_{1a}$) has a connectedness structure similar to the one of the KKT point set of an indefinite quadratic program a Hilbert
space~\cite[Proposition 3.8]{LTY_JOGO2023}. Note that Proposition~3.8 of \cite{LTY_JOGO2023} is an extension for Lemma~3.1 from~\cite{Luo_Tseng_1992}.}
\end{Remark}

\begin{Example}\label{Example2}
	{\rm Let $X, C, g, h$ be defined as in Example~\ref{Example1}. Formula~\eqref{S1_2} shows that the local solution set~${\mathcal S}_{1a}$ of~\eqref{DC_constraint_prob} is the union of three semi-closed generalized polydedral convex subsets of $C$, each of them is a connected component of ${\mathcal S}_{1a}$. In addition, the objective function $f:=g-h$ has a constant value on each one of these connected components. So, the example under our consideration can serve as a good illustration for Theorems~\ref{thm_structure1}--\ref{thm_const_value_2}.}
\end{Example}
  
\section{Solution Algorithms via Duality}\label{sect_5}
\markboth{\sc v. t. huong, d. t. k. huyen, and n.~d.~yen}{\sc  v. t. huong, d. t. k. huyen, and n.~d.~yen}
\setcounter{equation}{0}

The following theorem presents some conditions under which both primal problem and dual problem are gpdc programs in the sense of Definition~\ref{def_gpdc}.

\begin{Theorem}\label{Thm_duality_1} Let $g:X\to\overline{\mathbb {R}}$ be a lower semicontinuous, proper, convex function, $h:X\to\overline{\mathbb{R}}$ a proper generalized polyhedral convex function, and $C\subset X$ a nonempty closed convex set with $({\rm dom}\, g)\cap C\neq\emptyset$. Then, the optimal values of~\eqref{DC_constraint_prob} and~\eqref{dual_program_1} are equal. 
If, in addition, $g$ is a proper generalized polyhedral convex function and the set $C\subset X$ is generalized polyhedral convex, then both primal problem~\eqref{DC_constraint_prob} and dual problem~\eqref{dual_program_1} are gpdc programs, i.e., both problems have the same structure.
\end{Theorem}
\begin{proof} The first assertion follows from Corollary~\ref{Cor_duality}. Now, suppose that $g$ is a proper generalized polyhedral convex function and the set $C\subset X$ is generalized polyhedral convex. Then, $\delta_C$ is a proper generalized polyhedral convex function. The condition $({\rm dom}\, g)\cap C\neq\emptyset$ implies that $g+\delta_C$ is a proper generalized polyhedral convex function (see~~\cite[Theorem 3.7]{Luan_Yao_Yen_2018}).  By~\cite[Theorem 4.12]{Luan_Yao_Yen_2018}, the conjugate function of a proper generalized polyhedral convex function is a proper generalized polyhedral convex function. So, the second assertion is valid. $\hfill\Box$
\end{proof}

\medskip
The celebrated scheme for DC algorithms (called \textit{DCA scheme} for brevity) was given in~\cite{Tao_An_AMV97,Tao_An_SIOPT98} in a finite-dimensional setting. Recently, it has been reformulated in~\cite{LTY_JOGO2023} for a Hilbert space setting. We can describe the DCA for the problem~\eqref{DC_constraint_prob} in a locally convex Hausdorff topological vector space setting, which can be treated as the unconstrained DC problem~\eqref{DC_unconstraint_prob}, as follows.

\medskip
{\bf  DCA scheme}\\
{\it
	$\bullet$ Choose $x^0\in {\rm dom}\, (g+\delta_C)=({\rm dom}\, g)\cap C$.\\
	$\bullet$ For every integer $k\ge 0$, if $x^k$ has been defined and $\partial h(x^k)\neq\emptyset$, then select a vector
	\begin{equation}\label{existing_condition_of_y_k}
		\xi^k\in \partial h(x^k).
	\end{equation}
	$\bullet$ For every integer $k\ge 0$, if $\xi^k$ has been defined and $\partial (g+\delta_C)^*(\xi^k)\neq\emptyset$, then select a vector
	\begin{equation}\label{existing_condition_of_x_k}
		x^{k+1}\in \partial (g+\delta_C)^*(\xi^k).
\end{equation}} 
\begin{figure}[htp]
	\begin{center}
		\includegraphics[height=8.5cm,width=12cm]{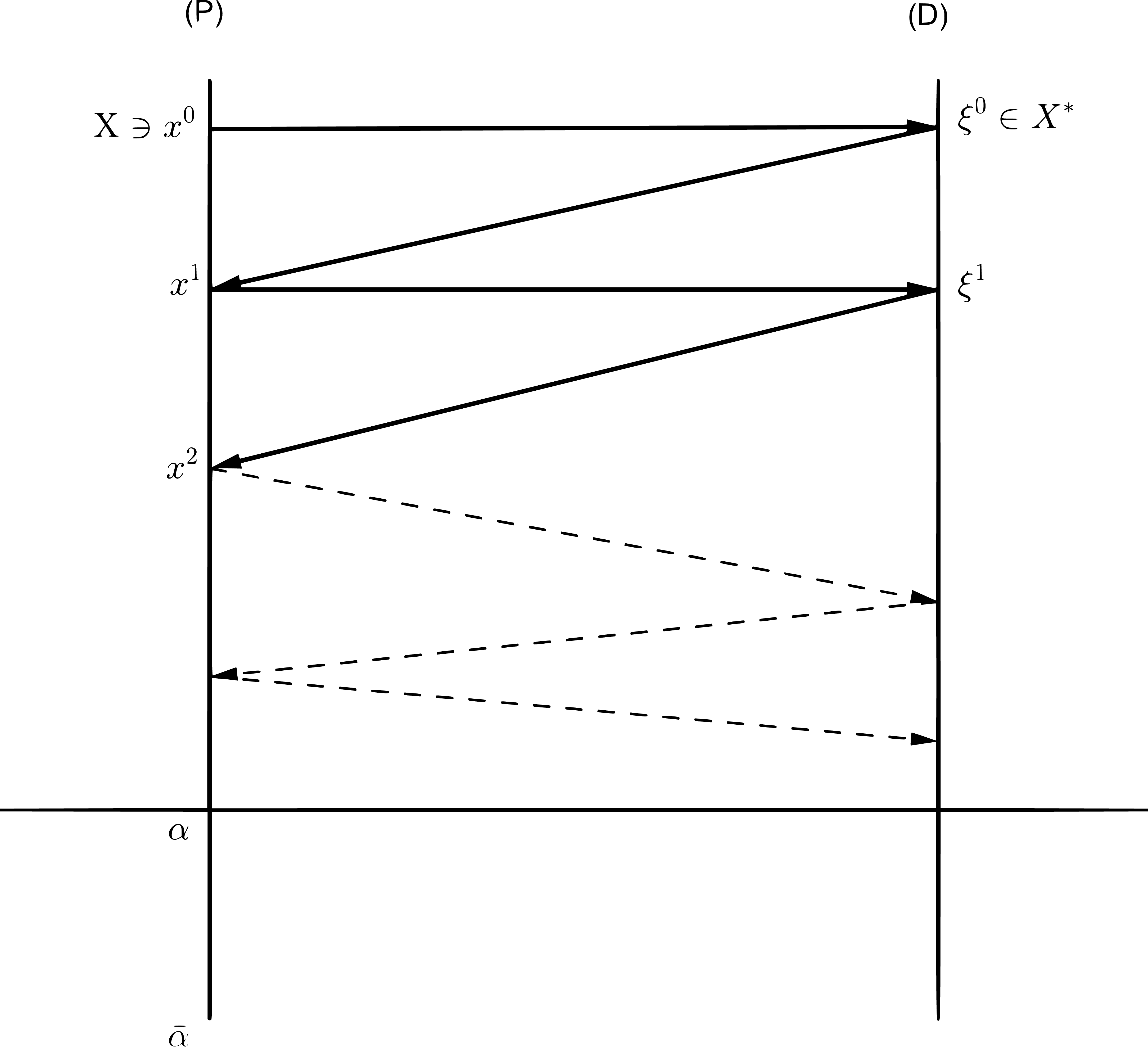}
	\end{center}
	\caption{DCA scheme}\label{Fig_2}
\end{figure}
On one hand, by Proposition~\ref{Equivalence_of_two_inclusions} we can rewrite the inclusion~\eqref{existing_condition_of_y_k} equivalently as
$x^k\in\partial h^*(\xi^k)$. The latter means that
$$h^*(x^*)-h^*(\xi^k)\geq \langle x^*-\xi^k,x^k\rangle\quad \forall x^*\in X^*.$$ So,
$$h^*(x^*)-\langle x^*,x^k\rangle\geq h^*(\xi^k)-\langle \xi^k,x^k\rangle\quad \forall x^*\in X^*.$$ Thus, $\xi^k$ is a solution to the convex optimization problem
\begin{equation}\label{convex_program_1}
	{\rm Minimize}\ \; h^*(x^*)-\langle x^*,x^k\rangle,\ \; x^*\in X^*.
\end{equation} On the other hand, condition~\eqref{existing_condition_of_y_k} is equivalent to the fulfillment of the requirement $x^k\in {\rm dom}\, h$ and the condition
$$h(x)-h(x^k)\geq \langle \xi^k,x-x^k\rangle\quad\; \forall x\in X,$$ which means that
\begin{equation}\label{condition_5.1} h(x)-\langle \xi^k,x\rangle\geq h(x^k)-\langle \xi^k,x^k\rangle\quad\; \forall x\in X.
\end{equation}

Applying Proposition~\ref{Equivalence_of_two_inclusions} once again, we can rewrite the inclusion~\eqref{existing_condition_of_x_k} equivalently as $\xi^k\in\partial (g+\delta_C)(x^{k+1})$. The latter signifies that $$x^{k+1}\in {\rm dom}\, (g+\delta_C)=({\rm dom}\, g)\cap C$$ and 
$$(g+\delta_C)(x)-(g+\delta_C)(x^{k+1})\geq \langle\xi^k,x-x^{k+1}\rangle\quad\; \forall x\in X.$$ Hence,
\begin{equation}\label{condition_5.2} g(x)-\langle \xi^k,x\rangle\geq g(x^{k+1})-\langle \xi^k,x^{k+1}\rangle\quad\; \forall x\in C.
\end{equation} This means that $x^{k+1}$ is a solution to the convex optimization problem
\begin{equation}\label{convex_program_2}
{\rm Minimize}\ \;  g(x)-\langle \xi^k,x\rangle,\ \; x\in C.
\end{equation} Substitute $x=x^{k+1}$ to the inequality in~\eqref{condition_5.1}, $x=x^k$ to the inequality in~\eqref{condition_5.2}, add the resulted inequalities side-by-side, and make a simple rearrangement to get $$g(x^k)-h(x^k)\geq g(x^{k+1})-h(x^{k+1}).$$ It follows that
\begin{equation}\label{decreasing}
	f(x^k)\geq f(x^{k+1})\quad\; (k=0,1,2,\ldots),
\end{equation} where $f=g-h$ is the objective function of~\eqref{DC_constraint_prob}.

We have thus proved the next result, which is originated from Pham Dinh and Le Thi's pioneering research works~\cite{Tao_An_AMV97,Tao_An_SIOPT98}, about distinguished features of the above infinite-dimensional DCA scheme.  

\begin{Theorem}\label{Thm_DCA1} The DCA sequences $\{x^k\}$ and $\{\xi^k\}$ constructed by the infinite-dimen\-sio\-nal \textbf{DCA scheme} have the following properties:
	\begin{itemize}
		\item[{\rm (a)}] $\{x^k\}\subset ({\rm dom}\, g)\cap ({\rm dom}\, h)\cap C$ and $\{\xi^k\}\subset \big({\rm dom}\, (g+\delta_C)^*\big)\cap {\rm dom}\, h^*$.
		\item[{\rm (b)}] At each step $k\geq 0$, the procedure~\eqref{existing_condition_of_y_k} to find $\xi^k$ can be realized either by computing the subdifferential $\partial h(x^k)$ of $h$ at the current iteration point $x^k$ or solving the convex optimization problem~\eqref{convex_program_1}, provided that the conjugate function $h^*$ can be defined explicitly.
		\item[{\rm (c)}] At each step $k\geq 1$, the procedure~\eqref{existing_condition_of_x_k} to find $x^{k+1}$ can be realized by solving the convex optimization problem~\eqref{convex_program_2} or computing the subdifferential $\partial (g+\delta_C)^*(\xi^k)$ of the conjugate function $(g+\delta_C)^*$ at the current iteration point $\xi^k$, provided that the conjugate function $(g+\delta_C)^*$ can be defined explicitly.
		\item[{\rm (d)}] The value of the objective function of~\eqref{DC_constraint_prob} decreases monotonically along the sequence $\{x^k\}$, i.e., the inequality in~\eqref{decreasing} is valid for all $k\in\mathbb N$.	
	\end{itemize}
\end{Theorem}

\begin{Example}\label{Example3}
	{\rm Consider the DC optimization problem~\eqref{DC_constraint_prob} with $X, C, g, h$ being defined as in Example~\ref{Example1}. Taking different initial points $$x^0\in {\rm dom}\, (g+\delta_C)=({\rm dom}\, g)\cap C=[-2,3]$$ and using the above DCA scheme, we obtain different DCA sequences.
	
1. For $x^0\in [-2,-1)$, since $\partial h(x^0)=\{-1\}$,~\eqref{existing_condition_of_y_k} implies that $\xi^0=-1$. Hence, substituting $k=0$ and $\xi^0=-1$ into the condition~\eqref{existing_condition_of_x_k}, which is equivalent to the requirement that $x^{k+1}$ is a solution of the convex optimization problem~\eqref{convex_program_2}, gives $x^1=-2$. Further computation shows that $x^k=-2$ for all $k\geq 2$. Thus, the DCA scheme yields the local solution $\bar x:=-2$.

2. For $x^0\in (1,3]$, since $\partial h(x^0)=\{1\}$,~\eqref{existing_condition_of_y_k} implies that $\xi^0=1$. Since the convex optimization problem~\eqref{convex_program_2} with $k=0$ has the unique solution $\hat x:=-1$, we get $x^1=3$. Further computation shows that $x^k=3$ for all $k\geq 2$. Thus, the DCA scheme yields the global solution $\hat x=3$.

3. For $x^0=-1$, since $\partial h(x^0)=[-1,0]$, by~\eqref{existing_condition_of_y_k} we can choose $\xi^0=0$. Since the convex optimization problem~\eqref{convex_program_2} with $k=0$ has the solution set $[-2,3]$, we can take $x^1=1$. Then, as $\partial h(x^1)=[0,1]$, by~\eqref{existing_condition_of_y_k} we can choose $\xi^1=0$. Since the solution of~\eqref{convex_program_2} with $k=1$ is $[-2,3]$, we can take  $x^2=-1$. Proceeding furthermore in the same manner, we obtain the iteration sequence $-1, 1, -1, 1,...$, which is divergent. Note that if we take $x^1\in [-2,-1)$, then the computation yields $x^k=-2$ for all $k\geq 2$, i.e., it leads to a local solution. Likewise, if we take $x^1\in (1,3]$, then the computation yields $x^k=3$ for all $k\geq 2$, i.e., the global unique solution is obtained.  

4. For $x^0\in (-1,1)$, an analysis similar to the one in the last case shows that we may obtain either a DCA sequence converging to the unique global solution $\hat x=3$, a local solution from  ${\mathcal S}_1\setminus {\mathcal S}$, a critical point from ${\mathcal S}_2\setminus {\mathcal S}_1$, or a divergent DCA sequence.
}
\end{Example}

To have cyclic DCA sequences, we can apply the idea of Pham Dinh and Le Thi~\cite[Section~4]{Tao_An_AMV97} in choosing a selection for each one of the subdifferential mappings $\partial h(.)$ and  $\partial (g+\delta_C)^*(.)$ to deal, respectively, with the conditions~\eqref{existing_condition_of_y_k} and~\eqref{existing_condition_of_x_k}.  

\medskip
First, suppose that the following assumption is fulfilled.

\begin{itemize}
	\item[{\rm \textbf{(A1)}}] \textit{$g:X\to\overline{\mathbb {R}}$ is a lower semicontinuous, proper, convex function, $h:X\to\overline{\mathbb{R}}$ a proper generalized polyhedral convex function, and $C\subset X$ a nonempty closed convex set with $({\rm dom}\, g)\cap C\neq\emptyset$.} 
	\end{itemize}

By~\textbf{(A1)}, $h$ is a proper generalized polyhedral convex function. Using~Theorem~4.14 and formula~(4.10) from~\cite{Luan_Yao_Yen_2018},  with $h$ playing the role of $f$, we can assert that there are finitely many nonempty generalized polyhedral sets in $X^*$, denoted by $P_1, P_2,\ldots, P_r$, such that for very $x\in {\rm dom}\,h$ there is a unique index $\ell$ from $P:=\{1,\ldots,r\}$ such that $\partial h(x)=P_\ell$.

Let $H:[C\cap {\rm dom}\,\partial h]\to X^*$, $x\mapsto H(x)$, be a selection of the subdifferential mapping $\partial h:[C\cap {\rm dom}\,\partial h]\rightrightarrows X^*$ and  $G_C:{\rm dom}\,\partial(g+\delta_C)^*\to X$, $\xi\mapsto G_C(\xi)$, be a selection of the subdifferential mapping $\partial (g+\delta_C)^*:{\rm dom}\,\partial(g+\delta_C)^*\rightrightarrows X$. To construct such a selection $H$, we pick from each set $P_\ell,$ where $\ell\in P$, one element $\eta^\ell$ and put $H(x)=\eta^\ell$ whenever $\partial h(x)=P_\ell$. With the help of the selections $H$ and $G_C$, the above DCA scheme has the next form.

\medskip
{\bf  DCA scheme 1}\\
{\it
	$\bullet$ Choose $x^0\in {\rm dom}\, (g+\delta_C)=({\rm dom}\, g)\cap C$.\\
	$\bullet$ For every integer $k\ge 0$, if $x^k\in C$ has been defined, then take
	\begin{equation}\label{xi_k}
		\xi^k=H(x^k).
	\end{equation}
	$\bullet$ For every integer $k\ge 0$, if $\xi^k$ has been defined and $\xi^k\in {\rm dom}\,\partial(g+\delta_C)^*$, then take
	\begin{equation}\label{x_k}
		x^{k+1}=G_C(\xi^k).
\end{equation}} 

Clearly, for a given initial point $x^0\in {\rm dom}\, (g+\delta_C)=({\rm dom}\, g)\cap C$, if the computation by DCA scheme~1 can be done for every step $k\in\mathbb N$, then we get a unique DCA sequence $\{x^k\}$ and a unique DCA sequence $\{\xi^k\}$ (otherwise, no DCA sequences are obtained). 

\begin{Theorem}\label{Thm_DCA2} Suppose that the assumption~{\rm \textbf{(A1)}} is satisfied and $\{x^k\}$ and $\{\xi^k\}$ are the DCA sequences constructed by \textbf{DCA scheme~1}. Then, there exist $\bar k\in\mathbb N$ and $p\in\mathbb N$ satisfying the following properties:
	\begin{itemize}
		\item[{\rm (a)}] From the index $\bar k$, both sequences $\{x^k\}$ and $\{\xi^k\}$ are cyclic with the period $p$, i.e., $x^{k+p}=x^{k}$ and $\xi^{k+p}=\xi^{k}$ for every $k\geq\bar k$. 
		\item[{\rm (b)}] From the index $\bar k$, the value of the objective function of~\eqref{DC_constraint_prob} is constant, i.e., $f(x^k)=f(\bar x^k)$ for every $k\geq\bar k$, where $f:=g-h$.
	\end{itemize}
\end{Theorem}
\begin{proof} From the construction of $H$ via $\eta^1,\ldots,\eta^r$ and~\eqref{xi_k} it follows that $$\xi^k=H(x^k)\in \left\{\eta^1,\ldots,\eta^r\right\}\quad (\forall k\in\mathbb N).$$ Then, by~\eqref{x_k} we have $x^k\in \left\{G_C(\eta^1),\ldots,G_C(\eta^r)\right\}$ for every $k\in\mathbb N$. Therefore, by the Dirichlet principle the must exist and $\ell\in P$ and a subsequence $\{x^{k'}\}$ of $\{x^k\}$ such that $x^{k'}=G_C(\eta^\ell)$ for all $k'$. Denote by $\bar k$ the smallest index $k'$ and choose $p\in\mathbb N$ such that $\bar k+p$ is the second-smallest index $k'$. In result, we get $x^{\bar k+p}=x^{\bar k}$. This and the rule~\eqref{xi_k} imply that $\xi^{\bar k+p}=\xi^{\bar k}$. Now, since the DCA sequence $\{x^k\}$ is uniquely defined by DCA scheme~1, one must have $x^{k+p}=x^{k}$ for every $k\geq\bar k$. Similarly, since the DCA sequence $\{\xi^k\}$ is also uniquely defined by the scheme, the equality $\xi^{k+p}=\xi^{k}$ holds for every $k\geq\bar k$. Thus,  assertion~(a) of the theorem has been proved.
	
To prove assertion~(b), we apply the last assertion of Theorem~\ref{Thm_DCA1} get~\eqref{decreasing}. Then, by the above cyclic property of $\{x^k\}$, we have   
$$f(x^{\bar k})\geq f(x^{\bar k+1})\geq \ldots \geq f(x^{\bar k+p})=f(x^{\bar k}).$$ It follows that 
$f(x^{\bar k})=f(x^{\bar k+1})= \ldots = f(x^{\bar k+p}).$ Therefore, $f(x^k)=f(\bar x^k)$ for every $k\geq\bar k$.

The proof is complete. $\hfill\Box$
\end{proof}

\medskip
Next, suppose that the following assumption is satisfied.

\begin{itemize}
	\item[{\rm \textbf{(A2)}}] \textit{$g:X\to\overline{\mathbb {R}}$ is a proper generalized polyhedral convex function, $h:X\to\overline{\mathbb{R}}$ a lower semicontinuous  convex function, and $C\subset X$ a nonempty generalized polyhedral set with $({\rm dom}\, g)\cap C\neq\emptyset$.}
\end{itemize}

By~\textbf{(A2)}, the conjugate function $(g+\delta_C)^*:X^*\to \overline{\mathbb R}$ is proper generalized polyhedral convex (see the proof of the second assertion of Theorem~\ref{Thm_duality_1}). Applying~\cite[Theorem~4.14]{Luan_Yao_Yen_2018} with $X^*$ playing the role of $X$ and $(g+\delta_C)^*$ playing the role of $f$ in that theorem, by~\cite[formula~(4.10)]{Luan_Yao_Yen_2018} we can find finitely many nonempty generalized polyhedral sets in $X$, denoted by $Q_1, Q_2,\ldots, Q_s$, such that for very $\xi\in {\rm dom}\,(g+\delta_C)^*$ there is a unique index $\ell$ from $S:=\{1,\ldots,s\}$ such that $\partial (g+\delta_C)^*(\xi)=Q_\ell$.

As before, let $H:[C\cap {\rm dom}\,\partial h]\to X^*$, $x\mapsto H(x)$, be a selection of the subdifferential mapping $\partial h:[C\cap {\rm dom}\,\partial h]\rightrightarrows X^*$ and  $G_C:{\rm dom}\,\partial(g+\delta_C)^*\to X$, $\xi\mapsto G_C(\xi)$, be a selection of the subdifferential mapping $\partial (g+\delta_C)^*:{\rm dom}\,\partial(g+\delta_C)^*\rightrightarrows X$. To construct such a selection $G_C$, it suffices to choose from each set $Q_\ell,$ where $\ell\in S$, one element $u^\ell$ and put $G_C(\xi)=u^\ell$ whenever $\partial (g+\delta_C)^*(\xi)=Q_\ell$. Suppose that this construction has been made. Thanks to the selections $H$ and $G_C$, the \textbf{DCA scheme} can be reformulated as the above \textbf{DCA scheme~1}.

\begin{Theorem}\label{Thm_DCA2a} Suppose that the assumption~{\rm \textbf{(A2)}} is fulfilled and $\{x^k\}$ and $\{\xi^k\}$ are the DCA sequences constructed by \textbf{DCA scheme~1}. Then, there exist $\bar k\in\mathbb N$ and $p\in\mathbb N$ satisfying the properties~{\rm (a)} and~{\rm (b)} in Theorem~\ref{Thm_DCA2}.
\end{Theorem}
\begin{proof} From~\eqref{x_k} we can deduce that $x^k\in \{u^1,\ldots,u^s\}$ for every $k\in\mathbb N$. So, there exists $\ell\in S$ and a subsequence $\{x^{k'}\}$ of $\{x^k\}$ such that $x^{k'}=u^\ell$ for all $k'$. Let $\bar k$ the smallest index $k'$ and $p\in\mathbb N$ such that $\bar k+p$ is the second-smallest index $k'$. Then, it holds that $x^{\bar k+p}=x^{\bar k}$. Combining this with the rule~\eqref{xi_k} yields $\xi^{\bar k+p}=\xi^{\bar k}$.
	
To show that the properties~{\rm (a)} and~{\rm (b)} in Theorem~\ref{Thm_DCA2} are valid for the chosen numbers $\bar k$ and $p$, we need only to repeat the arguments used in the final part of the preceding proof. $\hfill\Box$
\end{proof}

\section{Conclusions}\label{sect_Conclusions}
\markboth{\sc v. t. huong, d. t. k. huyen, and n.~d.~yen}{\sc  v. t. huong, d. t. k. huyen, and n.~d.~yen}
\setcounter{equation}{0}

Optimality conditions and DCA schemes for general DC optimization problems on locally convex Hausdorff topological vector space have been considered in this paper. When either the second component of the objective function is a generalized polyhedral convex function or the first component of the objective function is a generalized polyhedral convex function and the constraint set is generalized polyhedral convex, we have shown that sharper results on optimality conditions, the local solution set, the global solution set, and solution algorithms can be obtained. 

\section*{Acknowledgements}
\markboth{\sc v. t. huong, d. t. k. huyen, and n.~d.~yen}{\sc  v. t. huong, d. t. k. huyen, and n.~d.~yen}
\setcounter{equation}{0}

This research was supported by the project NCXS02.01/24-25 of Vietnam Academy of Science and Technology. The authors are indebted to the Vietnam Institute for Advanced Study in Mathematics for hospitality during their recent stay. The authors wish to thank Dr.~Nguyen Ngoc Luan for a useful remark which helped us to simplify the proof of Lemma~\ref{Lem3}. Duong Thi Kim Huyen would like to thank Phenikaa University for creating the favorable environment for research and teaching.


\begin{thebibliography}{99}
\markboth{\sc v. t. huong, d. t. k. huyen, and n.~d.~yen}{\sc  v. t. huong, d. t. k. huyen, and n.~d.~yen}
\setcounter{equation}{0}
		
	\bibitem{Bonnans_Shapiro_2000} Bonnans, J.F.,  Shapiro, A.:  {Perturbation Analysis of Optimization Problems}. Springer-Verlag, New York  (2000) 
	
	\bibitem{Contesse_1980} Contesse, L.: Une caract\'erisation compl\`ete des minima locaux en programmation quadratique. Numer. Math. \textbf{34}, 15--332 (1980)
	
	\bibitem{HangNTV_Yen_2016}   Hang, N.T.V., Yen, N.D.: On the problem of minimizing a difference of polyhedral convex functions under linear constraints. J. Optim. Theory  Appl. \textbf{171}, 617--642 (2016)
	
	\bibitem{Kelley_1975}	Kelley, J.L.: General Topology,
	Springer-Verlag, New York-Berlin (1975). [Reprint of the 1955 Edition published by Van Nostrand, New York]
		
	\bibitem{Ioffe_Tihomirov_1979}  Ioffe, A.D.,   Tihomirov, V.M.:  {Theory of Extremal Problems}. North-Holland Publishing Co., Amsterdam-New York (1979)
	
     \bibitem{An_Tao_2017} Le Thi, H.A., Pham Dinh, T.:  DC programming and DCA: thirty years of developments. 
     	Math. Program. \textbf{169},  5--68 (2018)
  
   \bibitem{An_Tao_2024} Le Thi, H.A., Pham Dinh, T.: Open issues and recent advances in DC programming and DCA. J. Global Optim. \textbf{88}, 533--590 (2024)
	
	\bibitem{LTY_JOGO2023}  Lim, Y.,  Tuan, H.N., Yen, N.D.: DC algorithms in Hilbert spaces and the solution of indefinite infinite-dimensional quadratic programs. J. Global Optim., DOI: 10.1007/s10898-023-01331-7. Published online: 03~October 2023.
	
	\bibitem{Luan_Yao_JOGO_2019}  Luan, N.N.,  Yao, J.-C.: Generalized polyhedral convex optimization problems. J. Global Optim. \textbf{75}, 789--811 (2019)
	
	\bibitem{Luan_Yao_Yen_2018}   Luan, N.N.,  Yao, J.-C., Yen, N.D.: On some generalized polyhedral convex constructions. Numer. Funct. Anal. Optim. \textbf{39}, 537--570 (2018) 
	
	\bibitem{Luan_Yen_2020}  Luan, N.N., Yen, N.D.: A representation of generalized convex polyhedra and applications. Optimization \textbf{69}, 471--492 (2020)
	
	\bibitem{Luan_Yen_2023}  Luan, N.N., Yen, N.D.: Strong duality and solution existence under minimal assumptions in conic linear programming. J. Optim. Theory Appl. \textbf{203}, 1083--1102 (2024)
	
	\bibitem{Luo_Tseng_1992}  Luo, Z.Q., Tseng, P.: Error bound and convergence analysis of matrix splitting algorithms for the affine variational inequality problem. SIAM J. Optim. \textbf{2}, 43--54 (1992) 
	
	\bibitem{Mordukhovich_Nam_ConvexAnalysis_2022}  Mordukhovich, B.S., Nam, N.M.: {Convex Analysis and Beyond, Vol. 1: Basic Theory}. Springer, Cham (2022)
	
    \bibitem{Tao_An_AMV97}  Pham Dinh, T., Le Thi, H.A.: Convex analysis approach to D.C. programming: Theory, algorithms and applications.  Acta Math. Vietnam. \textbf{22}, 289--355 (1997)
	
	\bibitem{Tao_An_SIOPT98}  Pham Dinh, T.,  Le Thi, H.A.:   A d.c. optimization algorithm for solving the trust-region subproblem. SIAM J. Optim. \textbf{8}, 476--505 (1998) 
	
    \bibitem{Polyakova_2011}  Polyakova, L.N.: On global unconstrained minimization of the difference of polyhedral functions.	J. Global Optim. \textbf{50}, 179--195 (2011)
    
    \bibitem{Rock_book_1970}  Rockafellar, R.T.:  Convex Analysis. Princeton University Press, Princeton, New Jersey (1970)
	
	\bibitem{Singer_1979} Singer, I.: A Fenchel-Rockafellar type duality theorem for maximization. Bull. Austral. Math. Soc. \textbf{20}, 193--198 (1979) 
	
	\bibitem{Singer_2006} Singer, I.:  {Duality for Nonconvex Approximation and Optimization}. Springer, New York (2006)
	
	\bibitem{Toland78}  Toland, J.F.: Duality in nonconvex optimization. J. Optim. Theory Appl. \textbf{66}, 399--415 (1978) 
		
	\bibitem{vom Dahl_Lohne_2020} vom Dahl, S., L\"ohne, A.: Solving polyhedral d.c. optimization problems via concave minimization. J. Global Optim. \textbf{78}, 37--47 (2020)
	
	\bibitem{Zheng_2009}  Zheng, X.Y.: Pareto solutions of polyhedral-valued vector optimization problems in Banach spaces. Set-Valued Anal. \textbf{17}, 389--408 (2009)
\end{thebibliography}
\end{document}